\documentclass{article}  

\usepackage{graphics} 
\usepackage{mathptmx} 
\usepackage{times} 
\usepackage{amsmath} 
\usepackage{amssymb}  
\usepackage{amsthm}
\usepackage{xcolor}
\usepackage{enumerate} 
\usepackage{mathtools}
\usepackage{dsfont}
\usepackage{comment}

\usepackage[font={small}]{caption}

\newcommand{\redd}{\hfill $\lozenge$}

\usepackage{tikz}
\usetikzlibrary{arrows,shapes}

\usepackage[ruled,linesnumbered]{algorithm2e}

\SetCommentSty{mycommfont}
\let\oldnl\nl
\newcommand{\nonl}{\renewcommand{\nl}{\let\nl\oldnl}}



\newcommand{\change}[1]{{  #1}}

\newtheorem{theorem}{Theorem}
\newtheorem{lemma}{Lemma}
\newtheorem{remark}{Remark}

\newtheorem{assumption}{Assumption}

\DeclareMathOperator*{\minimize}{minimize}

\DeclareMathOperator*{\st}{subject\ to}

\title{\LARGE \bf
A Structured Optimal Controller with Feed-Forward for Transportation
}

\author{Martin Heyden, Richard Pates and Anders Rantzer
\footnote{This work was supported by the Swedish Foundation for Strategic Research through the project SSF {RIT15-0091} SoPhy.\newline
\indent The authors are members of the LCCC Linnaeus Center and the ELLIIT Excellence Center at Lund University.\newline
\indent The authors are with the Department of Automatic Control, Lund University, Box 118, {SE-221 00 Lund}, Sweden. \newline
\indent © 2021 IEEE.  Personal use of this material is permitted.  Permission from IEEE must be obtained for all other uses, in any current or future media, including reprinting/republishing this material for advertising or promotional purposes, creating new collective works, for resale or redistribution to servers or lists, or reuse of any copyrighted component of this work in other works.
}%
}

 \usepackage{hyperref}
 \hypersetup{
   colorlinks,           
   bookmarksnumbered,
   linkcolor={blue!50!black},
   citecolor={blue!50!black}, 
   urlcolor={blue!50!black}
}
\usepackage[all]{hypcap}

\usepackage[capitalise,noabbrev,nameinlink]{cleveref}
\crefname{assumption}{Assumption}{Assumptions}

\begin{document}

\maketitle

\begin{abstract}
We study an optimal control problem for a simple transportation model on a \change{path} graph. We give a closed form solution for the optimal controller, which can also account for planned disturbances using feed-forward. The optimal controller is highly structured, which allows the controller to be implemented using only local communication, conducted through two sweeps through the graph.
\end{abstract}

\section{Introduction}
In this paper we study a simple Linear Quadratic control transportation problem on a network. Such problems have well known solutions based on the Riccati equation \cite{kalman1960}. This gives a static feedback law 
$$
u = Kx,
$$
where $u$ is the input to the system, $x$ is the state of the system and $K$ is a matrix with real entries. This matrix is in general dense. This is undesirable in large-scale problems, since it implies that measurements from the entire network are required to compute the optimal inputs at every node. Furthermore a centralized coordinator with knowledge of the entire system is required to determine the matrix $K$, and a complete redesign will be required in response to any changes in the network.

These factors have led to the development of a range of general purpose methods for structured control system design. Some notable themes include the notion of Quadratic Invariance \cite{RL06,LL11}, System Level Synthesis \cite{WMD18}, and the use of large-scale optimization techniques (e.g. \cite{LFJ13}). A downside with these approaches is that they improve scalability at the expense of performance. That is they search over families of controllers that exclude the dense optimal controller for \eqref{eq:problem}. While in comparison with the alternative this may be an acceptable trade-off, it implicitly assumes that just because the feedback law is dense, it cannot be efficiently implemented.

The main result of this paper is to show that the simple structure in our problem allows the optimal control law  to be computed and implemented in a simple and scalable manner. The resulting control actions are the same as those from a Riccati approach, and could in principle be calculated that way. However, there are extra structural features in the control law that are obscured by the resulting dense feedback matrix representation, and it is not obvious how to exploit these to give a scalable implementation from the gain matrix obtained from the Riccati equation. 

\subsection{Problem Formulation}
We consider the problem of transportation and production of goods on a directed path graph
 with vertices $v_1,v_2,\ldots{},v_N$ and directed edges $(e_N,e_{N-1}),\ldots{},(e_2,e_1)$.
The dynamics are given by
\begin{equation}\label{eq:dynamics}
  z_i[t+1]  = z_i[t] - u_{i-1}[t] + u_{i}[t-\tau_{i}] + v_i[t] + d_i[t].
\end{equation}
All the variables are considered to be defined relative to some equilibrium. In the above $z_i[t] \in \mathds{R}$ is the quantity in node $i$ at time $t$. The system can be controlled using the variables $u_i[t]\in \mathds{R}$ and $v_i[t]\in \mathds{R}$. The variable $u_i[t]$ denotes the amount of the quantity that is transported from node $i+1$ to node $i$ (again relative to some equilibrium flows), and the transportation takes $\tau_{i}$ time units. For the last node $N$ it is assumed that $u_N[t] = 0$ for all $t$. The variable $v_i[t]$ denotes the flexible production or consumption of the quantity at the \emph{i}th node. Finally $d_i[t]\in \mathds{R}$ is the fixed production/consumption at the \emph{i}th node. This will be treated like a forecast, or \emph{planned disturbance}, that is known to the designer, but cannot be changed. This model could for instance describe a water irrigation network \cite{cantoni2007control} or a simple supply chain system \cite{inventory}. A state-space representation for \eqref{eq:dynamics} can be obtained by setting $z_i[t]$, $u_i[t-\delta],\ 1\leq\delta\leq\tau_i$ to equal the system state.

The goal is to optimally operate this network around some equilibrium point. The performance is measured by the cost of deviating from the equilibrium levels $q_iz_i^2$ and the cost of  the variable production $r_iv_i^2$,  where $q$ and $r$ are strictly positive constants.
We thus consider the following linear quadratic control problem on a graph with $N$ nodes,
\begin{equation}\label{eq:problem}\begin{aligned}
  \minimize_{z,u,v} \quad & \sum_{t=0}^\infty\sum_{i=1}^N \big(q_iz_i[t]^2 + r_iv_i[t]^2 \big)\\
  \st \quad & \text{dynamics in \eqref{eq:dynamics}}\\
  & z[0], d_i[t].
\end{aligned}\end{equation}
Note that there is no penalty on the internal flows $u_i$. This can for example be motivated by the transportation costs already being covered by the costs of the nominal flows (or in the case of water irrigation networks that gravity does the moving). This problem is in effect a dynamic extension of the types of scheduling problems considered in transportation networks \cite{ahuja1995applications}, and could be used to compliment such approaches by optimally adjusting a nominal schedule in real time using the feedback principle.

A similar problem has been studied in a previous paper \cite{heyden2018structured}. However we give several important extensions, in that we allow for non-homogeneous delays, production in every node and optimal feed-forward for planned disturbances. Allowing for non homogeneous delays is important as that will be the case for almost all applications. Taking planned disturbances into account allows for increased performance whenever such disturbances can be forecast.  Finally, allowing for some variation in the consumption $v_i$ for each node will also generally increase performance whenever such variation is possible. The effect the feed-forward of planned disturbances can have on the controller performance is illustrated in \cref{fig:example}, where we see that the controller with feed-forward anticipates the action of the disturbances, allowing the effect to be better spread through the graph and the node levels to be more tightly regulated. This results in a significant improvement in performance.

\begin{figure}[]
  \centering
  \includegraphics{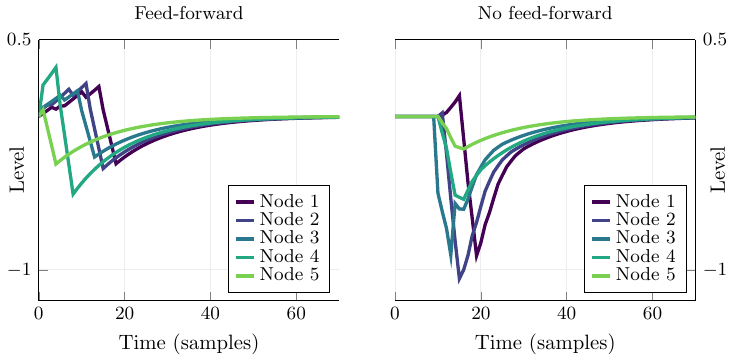}
  \caption{Example of the effect of feed forward. The graph has five nodes and transportation delay $\tau_1 = 3$, $\tau_2 = 2$, $\tau_3 = 5$ and $\tau_4 =4$. There is a disturbance in node three from time 10 to 13 and in node 2 from time 12 to 15. We can see that the feed-forward manages to handle the disturbances better by spreading out their effect throughout the graph.
  To quantify the difference one can consider the cost in \eqref{eq:problem}, which is $3.11$ with feed-forward and 11.35 without feed-forward.}
  \label{fig:example}
\end{figure}


\subsection{Result Preview}
The key structural feature that we identify in the optimal control law for \eqref{eq:problem} is that optimal inputs can be computed recursively by two sweeps through the graph (even though the control law that would be obtained from the Riccati equation would be dense). More specifically, two intermediate variables local to the \emph{i}th node $\delta_i[t]$ and $\mu_i[t]$ can be computed recursively through relationships on the form
\[
\begin{aligned}
  \delta_i[t]&=f(\text{local\_variables},\delta_{i-1}),\\
  \mu_i[t]&=g(\text{local\_variables},\mu_{i+1}),
\end{aligned}
\]
from which the optimal inputs $u_i[t]$ and $v_i[t]$ can be calculated based only on local variables. Conceptually this step is rather similar to solving a sparse system of equations with the structure of a directed path graph using back substitution.

The details are given in \cref{alg:com_gen}, and this process is illustrated in \cref{fig:com_ill}. This allows the optimal inputs to be computed by sweeping once through the graph from the first node to the final node to compute the $\delta$'s, and once from the final node to the first to compute the $\mu$'s. Both sweeps can be conducted in parallel. This represents a sort of middle ground between centralised control and decentralised control, in which global optimality is preserved whilst only requiring distributed communication. The price for this is that the sweep through the whole graph must be completed before the inputs can be applied, but for systems with reasonably long sample times (which is likely true in transportation or irrigation systems) this seems a modest price to pay. Interestingly the controller parameters can be computed in a similar distributed manner, allowing the controller to also be synthesised in a simple and scalable manner. This is shown in \cref{alg:init}.

\section{Results}

In this section we present two algorithms that together allow for the solution of \eqref{eq:problem}. The first of these algorithms computes the parameters of a highly structured control law for solving \eqref{eq:problem}, whereas the second shows that the control law has a simple distributed implementation. These features will be discussed in \cref{sec:impl}. In this section we will demonstrate that under suitable assumptions on the planned disturbances $d_i[t]$, \cref{alg:init,alg:com_gen}  give the optimal solution to \eqref{eq:problem} . This constitutes the main theoretical contribution of the paper.

In the absence of the planned disturbances (i.e. with $d_i[t]\equiv{}0$), \eqref{eq:problem} is an infinite horizon LQ problem in standard form. It is of course highly desirable in applications to be able to include information about upcoming disturbances in the synthesis of the control law. However if we are given an infinite horizon of disturbances, \eqref{eq:problem} is no longer tractable. For the theoretical perspective, it turns out that the suitable assumption on the horizon length is as follows:
\begin{assumption}\label{assump:horizon}
Let the aggregate delay $\sigma_k$ in \eqref{eq:dynamics} be
\begin{equation*}
  \sigma_k = \sum_{i=1}^{k-1} \tau_i.
\end{equation*}
Given a horizon length $H\geq0$, assume that
$d_i[t] = 0$ for all $t> H+(\sigma_N-\sigma_i)$ and for all $1\leq i \leq 
N$. 
\end{assumption}

Observe that if $d_i[t] = 0$ for all $t>H$ then \cref{assump:horizon} holds. Thus the assumption captures the natural notion of having a finite horizon $H$ of information about the disturbances $d_i[t]$ available when constructing the control input. 
In \cref{sec:simulations} we will investigate how the length of the horizon affects the performance of the controller.

We will now state the main results of this paper. The following theorem shows that \eqref{eq:problem} can be solved by running two simple algorithms; one for calculating all the necessary parameters, and one for computing the optimal inputs. Both algorithms can be implemented using only local communication as discussed in \cref{sec:impl}. A graphical illustration of the implementation of \cref{alg:com_gen}, which is the algorithm used for the on-line implementation, can be found in \cref{fig:com_ill}. 

\begin{theorem}\label{thm:gen}
Let
\[
D_i[t] = \sum_{j=1}^i d_j[t-\sigma_j].
\]
 Assume that $H$ and $d_j[t]$ satisfy \cref{assump:horizon}.  Then the optimal inputs $u_i[t]$ and $v_i[t]$ for the problem in \eqref{eq:problem} are given by running \cref{alg:com_gen} with the parameters calculated by \cref{alg:init}.
\end{theorem}
\begin{proof}
See the appendix. A sketch of the proof can be found in \cref{sec:proof_idea}.
\end{proof}
\begin{remark}
    In most cases the choice of $H$ can be made without considering its effect on the controller implementation, and can instead be chosen based only on the nodes' ability to forecast their disturbances. In applications it would also be natural to incorporate new information on upcoming disturbances in a receding horizon fashion.
    This will be further discussed in \cref{sec:recedinghorizon}. \redd
\end{remark}

\begin{remark}
There is an asymmetry in \cref{assump:horizon} in that $(\sigma_N-\sigma_i)$ grows as $i$ decreases from $N$ to $1$. This means that this assumption allows nodes further down the graph to have longer horizons of planned disturbances. Of course there is no reason to believe that these nodes are better at predicting their disturbances. It is just that the derived theory can handle those disturbances in a straightforward manner since the optimal controller lumps the disturbances into time shifted sums, with a time shift proportional to $\sigma_i$. \redd 
\end{remark}

\begin{algorithm}[]
\caption{Computation of control parameters.}
\label{alg:init}
\DontPrintSemicolon
\SetNoFillComment
\SetAlgoLined
\KwIn{$q_i$, $r_i$, $\tau_i$, $H$}
\KwOut{$\gamma_i$, $g_i(j)$, $P_i(\tau_i,1)$, $h_i$, $\phi_i(\Delta)$, $a_i$, $c_i$}
\vspace{-7pt}
\nonl \hrulefill  \\
\vspace{-5pt}
\tcc{First Sweep, upstream direction}
$\gamma_1 = q_1, \quad \rho_1 = r_1$ \tcp*[r]{initialize first node}
\textbf{send} $\gamma_1$ and $\rho_1$ to upstream neighbor\\
\For {node i = 2:N} { 
  $\gamma_i = \frac{\gamma_{i-1}q_i}{\gamma_{i-1} + q_i}, \quad \rho_i = \frac{\rho_{i-1}r_i}{\rho_{i-1} + r_i}$\\
  \textbf{send} $\gamma_i$ and $\rho_i$ to upstream neighbor
}
\vspace{-5pt}
\nonl \dotfill\\
\vspace{-3pt}
\tcc{Second Sweep  Downstream direction}
$X_N(H+2) =  -\frac{\gamma_N}{2}+\sqrt{\gamma_N\rho_N+\frac{\gamma_N^2}{4}}$ \\
\For {node i = N:1} {
  $X_i(\tau_i)  = \frac{\rho_i(X_{i+1}(1)+\gamma_{i})}{X_{i+1}(1)+\gamma_{i} + \rho_i}$ \tcp*{Not for node N} 
  $X_i(t-1)  = \frac{\rho_i(X_i(t)+\gamma_i)}{X_i(t) + \gamma_i+\rho_i}$,  \tcp*{$1\leq t-1\leq\tau_i-1$ or for i =N, $1\leq t-1 \leq H+1$}
  $g_i(i) = \frac{X_i(i)}{X_i(i)+\gamma_i},\quad 2\leq i \leq \tau_i$ \\
  $g_{i+1}(1) = \frac{X_{i+1}(1)}{X_{i+1}(1)+\gamma_{i}}$\\
  $b_i = g_{i+1}(1)\prod_{j=2}^{\tau_i}g_i(j)$ \\
  \textbf{send} $X_i(\tau_i)$, $b_i$ to downstream neighbor.\\
  \tcp{for $1\leq l,m \leq \tau_i$}
  $P_i(1,m)  = \frac{X_i(1)}{\rho_{i}}$ \\
  $P_i(l,m)  = (1-\frac{X_i(l)}{\rho_{i}})g_i(l)P_i(l-1,m) + \frac{X_i(l)}{\rho(i)}, \quad l\leq m$ \\
  $P_i(l,m)  = (1-\frac{X_i(l)}{\rho_{i}})P_i(l-1,m) + \frac{X_i(l)}{\rho(i)}, \quad l> m$ \\
}
\vspace{-5pt}
\nonl \dotfill \\
\vspace{-3pt}
\tcc{Third Sweep, upstream direction}
$h_1 = P_1(\tau_1,\tau_1)g_2(1)$\\
\textbf{send} $h_1$ to upstream neighbor.\\
\For {node i= 2:N-1} {
  $h_i = (1-P_i(\tau_i,1))b_ih_{i-1} + P_i(\tau_i,\tau_i)g_{i+1}(1)$ \\
  \textbf{send} $h_i$ to upstream neighbor.
}
\vspace{-5pt}
\nonl\dotfill \\
\vspace{-3pt}
\tcc{Some final local Calculations}
\tcp{For $1\leq \Delta \leq \tau_i$, Empty Product, $\prod_{j=2}^1$ = 1}
$\phi_i(\Delta) = \left(1-P_i(\tau_i,\Delta) -(1-P_i(\tau_i,1))h_{i-1}{\prod_{j=2}^{\Delta +1}}g_k(j) \right)$\\
$a_i  = \frac{X_i(1)}{r_i}+\frac{\gamma_i}{q_i}(1-\frac{X_i(1)}{\rho_i})$\\
$c_i  = -\Big(\frac{X_i(1)}{r_i} - \frac{\gamma_iX_i(1)}{q_i\rho_i} \Big)(1-h_{i-1}) + \frac{\gamma_i}{q_i}h_{i-1}$
\end{algorithm}
\SetInd{0.5em}{0.5em}

\begin{algorithm}[h!]
\caption{Distributed Controller Implementation.}
\label{alg:com_gen}
\SetNoFillComment
\KwIn{$z_i[t]$, $u_i[t-(\tau_i-\Delta)]$, $d_i[t]$, $D_i[t+\sigma_i +\Delta]$}
\KwOut{$u_i[t]$,$v_i[t]$}
\vspace{-0.25cm}
\nonl \hrulefill  \\
\vspace{-0.15cm}
\tcp{Let $\tau_N = H + 1$.}
\tcc{Upstream sweep - Done in parallel with downstream sweep}
\For{node i = 1:N} {
$\Phi_i[t] = \phi_i(1)z_i[t]  +$\\
\nonl$ \ \ \ \ \sum_{\Delta=0}^{\tau_i-1} \phi_i(\Delta+1)\Big(u_i[t-(\tau_i-\Delta)]
 + D_i[t+\sigma_i+\Delta]\Big)$\\
$\delta_i[t]  = \Phi_i[t] + (1-P_{i}(\tau_{i},1))\delta_{i-1}[t]$\\
\textbf{send} $\delta_i[t]$ upstream }
\vspace{-0.25cm}
\nonl \dotfill\\
\vspace{-0.15cm}
\tcc{Downstream sweep - Done in parallel with upstream sweep}
\For{node i = N:1} {
\tcp{Empty Product, $\prod_{j=2}^1$ = 1}
$\pi_i[t] = z_i[t] + \sum_{\Delta=0}^{\tau_i-1} \Big(u_i[t-(\tau_i-\Delta)] + D_i[t+\sigma_i+\Delta]\Big)\prod_{j=2}^{\Delta+1}g_i(j)$\\
$\mu_i[t] = \pi_i[t] + b_i\mu_{i+1}[t]$\\
\textbf{send} $\mu_i[t]$ downstream}
\vspace{-0.25cm}
\nonl \dotfill\\
\vspace{-0.15cm}
\tcc{Calculate outputs}
$u_{i-1}[t] = (1-\frac{\gamma_i}{q_i})\big(z_i[t]+u_i[t-\tau_i]+D_{i}[t+\sigma_i]\big)$ \\
\nonl  $\qquad \qquad \qquad \qquad \qquad -a_i\delta_{i-1}[t] + c_i\mu_i[t] + d_i[t] -D_{i}[t+\sigma_i]$\\
$v_i[t] = -\frac{X_i(1)}{r_i}\Big(\delta_{i-1}[t]  + (1-h_{i-1})\mu_i[t]\Big)$
\end{algorithm}

\section{Implementation}\label{sec:impl}

\begin{figure*}
\centering
\begin{tikzpicture}[scale=0.95,->,shorten >=1pt,auto,node distance=1.0cm,
                    thick,node/.style={circle,draw,radius = 1cm}]
\tikzstyle{transp} = [circle,dashed,draw,radius =1cm];   
        \def\nodey{5};
        \def\texty{5.9};
        \def\deltax{-3.6};
        \def\phix{-1.8};
        \def\pix{1.8};
        \def\mux{3.6};
        \def\dy{0.9}

\tikzstyle{transp2} = [circle,draw,inner sep=1.5pt,fill]
\tikzstyle{agg} = [rectangle,draw]
\tikzstyle{point} = [circle,draw,fill=black];
\node[node] at (0,0) (n1)  {$z_1$} ;
\node[agg] at (\pix,0) (pi1) {$\pi_1$};
\node[agg] at (\mux,0) (mu1) {$\mu_1$};
\node[agg] at (\phix,0) (phi1) {$\Phi_1$};
\node[agg] at (\deltax,0) (delta1) {$\delta_1$}; 
\draw[ultra thick, color = blue] (n1) -- (pi1);
\draw[ultra thick, color = blue] (pi1) -- (mu1);
\draw[ultra thick, color = red] (n1) -- (phi1);
\draw[ultra thick, color = red] (phi1) to (delta1);

\node[transp] at (0,\dy) (p1-1) {0};
\node[transp] at (0,2*\dy) (p1-2) {1};
\node[node] at (0,3*\dy) (n2) {$z_2$};
\draw[ultra thick, color = blue] (p1-1) to (pi1);
\draw[ultra thick, color = blue] (p1-2) to (pi1);
\draw[ultra thick, color = red] (p1-1) to (phi1);
\draw[ultra thick, color = red] (p1-2) to (phi1);

\node[agg] at (\pix,3*\dy) (pi2) {$\pi_2$};
\node[agg] at (\mux,3*\dy) (mu2) {$\mu_2$};
\node[agg] at (\phix,3*\dy) (phi2) {$\Phi_2$};
\node[agg] at (\deltax,3*\dy) (delta2) {$\delta_2$}; 
\draw[ultra thick, color = blue] (mu2) to (mu1);
\draw[ultra thick, color = red] (delta1) to (delta2);
\draw[ultra thick, color = blue] (n2) to (pi2);
\draw[ultra thick, color = red] (n2) to (phi2);
\draw[ultra thick, color = blue] (pi2) to (mu2);
\draw[ultra thick, color = red] (phi2) to (delta2);

\node[transp] at (0,4*\dy) (p2-1) {0};
\node[node] at (0,5*\dy) (n3) {$z_3$};
\draw[ultra thick, color = blue] (p2-1) to (pi2);
\draw[ultra thick, color = red] (p2-1) to (phi2);

\node[agg] at (\pix,5*\dy) (pi3) {$\pi_3$};
\node[agg] at (\mux,5*\dy) (mu3) {$\mu_3$};
\node[agg] at (\phix,5*\dy) (phi3) {$\Phi_3$};
\node[agg] at (\deltax,5*\dy) (delta3) {$\delta_3$}; 
\draw[ultra thick, color = blue] (mu3) to (mu2);
\draw[ultra thick, color = red] (delta2) to (delta3);
\draw[ultra thick, color = blue] (n3) to (pi3);
\draw[ultra thick, color = red] (n3) to (phi3);
\draw[ultra thick, color = blue] (pi3) to (mu3);
\draw[ultra thick, color = red] (phi3) to (delta3);
\node at (0,6*\dy) (inv1) {};
\node at (-2,6*\dy) (inv1d) {};
\node at (2,6*\dy) (inv1u) {};

\node[transp] at (0,6*\dy) (p3-1) {  0  };
\draw[ultra thick, color = red] (p3-1) to (phi3);
\draw[ultra thick, color = blue] (p3-1) to (pi3);

\node[node] at (0,7*\dy) (zN-1) {$z_{4}$};
\node[agg] at (\pix,7*\dy) (piN-1) {$\pi_{4}$};
\node[agg] at (\mux,7*\dy) (muN-1) {$\mu_{4}$};
\node[agg] at (\phix,7*\dy) (phiN-1) {$\Phi_{4}$};
\node[agg] at (\deltax,7*\dy) (deltaN-1) {$\delta_{4}$}; 
\draw[ultra thick, color = red] (delta3) to (deltaN-1);
\draw[ultra thick, color = blue] (muN-1) to (mu3);
\draw[ultra thick, color = blue] (zN-1) to (piN-1);
\draw[ultra thick, color = red] (zN-1) to (phiN-1);
\draw[ultra thick, color = blue] (piN-1) to (muN-1);
\draw[ultra thick, color = red] (phiN-1) to (deltaN-1);

\node[node] at (0,10*\dy) (zN) {$z_{5}$};
\node[agg] at (\pix,10*\dy) (piN) {$\pi_{5}$};
\node[agg] at (\mux,10*\dy) (muN) {$\mu_{5}$};
\node[agg] at (\phix,10*\dy) (phiN) {$\Phi_{5}$};
\node[agg] at (\deltax,10*\dy) (deltaN) {$\delta_{5}$}; 
\draw[ultra thick, color = red] (deltaN-1) to (deltaN);
\draw[ultra thick, color = blue] (muN) to (muN-1);
\draw[ultra thick, color = blue] (zN) to (piN);
\draw[ultra thick, color = red] (zN) to (phiN);
\draw[ultra thick, color = blue] (piN) to (muN);
\draw[ultra thick, color = red] (phiN) to (deltaN);
\node[transp] at (0,8*\dy) (pN-1) {0};
\node[transp] at (0,9*\dy) (pN-2) {1};
\draw[ultra thick, color = blue] (pN-1) to (piN-1);
\draw[ultra thick, color = blue] (pN-2) to (piN-1);
\draw[ultra thick, color = red] (pN-1) to (phiN-1);
\draw[ultra thick, color = red] (pN-2) to (phiN-1);

\node[node] at (\nodey,10*\dy) (gn5) {\small $z_5[t]$};
\node[transp2] at (\nodey,9*\dy) (gt42) {};
\node[] at (\texty,9*\dy) {$u_4[t-1]$};
\node[transp2] at (\nodey,8*\dy) (gt41) {}; 
\node[node] at (\nodey,7*\dy) (gn4) {\small $z_4[t]$};
\node[] at (\texty,8*\dy) {$u_4[t-2]$};
\node[transp2] at (\nodey,6*\dy) (gt31) {};
\node[] at (\texty,6*\dy) {$u_3[t-1]$};
\node[node] at (\nodey,5*\dy) (gn3) {\small $z_3[t]$};
\node[transp2] at (\nodey,4*\dy) (gt21) {};
\node[] at (\texty,4*\dy) {$u_2[t-1]$};
\node[node] at (\nodey,3*\dy) (gn2) {\small $z_2[t]$};
\node[transp2] at (\nodey,2*\dy) (gt12) {};
\node[] at (\texty,2*\dy) {$u_1[t-1]$};
\node[transp2] at (\nodey,1*\dy) (gt11) {};
\node[] at (\texty,\dy) {$u_1[t-2]$};
\node[node] at (\nodey,0) (gn1) {\small $z_1[t]$};
\draw[ultra thick] (gn5) -> (gt42);
\draw[ultra thick] (gt42) -> (gt41);
\draw[ultra thick] (gt41) -- (gn4);
\draw[ultra thick] (gn4) -- (gt31);
\draw[ultra thick] (gt31) -- (gn3);
\draw[ultra thick] (gn3) -- (gt21);
\draw[ultra thick] (gt21) -- (gn2);
\draw[ultra thick] (gn2) -- (gt12);
\draw[ultra thick] (gt12) -- (gt11);
\draw[ultra thick] (gt11) -- (gn1);

\end{tikzpicture}
\caption{
Illustration of the structured approach for calculating the optimal inputs using \cref{alg:com_gen}, for a 5 node example with $\tau_1 = 2$, $\tau_2 = 1$, $\tau_3 = 1$, $\tau_4 = 2$. The graph at the right of the figure illustrates the underlying dynamics of the network as in \eqref{eq:dynamics}. The left part of the figure illustrates the structure of the computations required to compute the optimal control according to \cref{alg:com_gen}. The solid circles corresponds to node states and dashed circles the quantities in transit. The number in each dashed circle denotes the value of $\Delta$, which then maps to $u_k[t-(\tau_k-\Delta)]$. The rectangles indicates the different intermediate calculations needed to determine the variables required to compute the optimal inputs (lines 2--3 and 7--8 in \cref{alg:com_gen}). These are horizontally aligned with the location in the network where they could be locally performed. The  arrows indicate information flow. Each intermediate can be calculated using only the quantities from the incoming arrows. An upstream sweep is performed (the red arrows) in order to calculate the variables $\delta_i[t]$. The local intermediate $\Phi_i[t]$ variables are calculated (line 2), and then aggregated into the $\delta_i[t]$'s (line 3), which are sequentially passed up the graph. For the downstream sweep the local $\pi_i[t]$ variables are calculated (line 7) and aggregated into the $\mu_i[t]$'s (line 8). Both sweeps can be conducted in parallel, and once they have completed, the optimal inputs for the \emph{i}th node can be determined using the variables at the \emph{i}th location according to lines 11 and 12 in \cref{alg:com_gen}.
}
\label{fig:com_ill}
\end{figure*}

In this section we will discuss the structure in \cref{alg:init,alg:com_gen}, and explain how they can be used to implement an optimal feedback control law for solving \eqref{eq:problem} in a distributed
  manner. In both cases the order in which the computations occur is highly structured. This is illustrated for \cref{alg:com_gen}, which is the algorithm that must be run to compute the control inputs, in \cref{fig:com_ill}.  Matlab code for using these algorithms to calculate the optimal control inputs is available at github\footnote{\texttt{https://github.com/Martin-Heyden/cdc-letters-2021}}, as well as code to verify that \cref{thm:gen} holds numerically. We will also discuss how to calculate $D_i$ and incorporate updates to the planned disturbances in a receding horizon style in \cref{alg:com_D}. 

\subsection{Algorithms 1 and 2, and the Optimal Control Law}

The problem in \eqref{eq:problem} is at its heart an LQ problem, and the optimal controller is given by a static feedback law. The corresponding feedback matrix is generally dense, and that is the case for \eqref{eq:problem} as well. However certain special structural features of the process \eqref{eq:dynamics} are inherited by the optimal control law. It is these features we exploit to give a scalable implementation in \cref{alg:init,alg:com_gen}, which we will now discuss. 

In terms of the algorithm variables, the optimal node production $v_i[t]$ for \eqref{eq:problem} is given by
\begin{equation}\label{eq:opt_v}
  v_i[t] = -\frac{X_i(1)}{r_i}\Big(\delta_{i-1}[t]  + (1-h_{i-1})\mu_i[t]\Big),
\end{equation}
and the optimal 
internal flows $u_i[t]$ are given by
\begin{multline}\label{eq:opt_u}
u_{i-1}[t] = (1-\frac{\gamma_i}{q_i})\big(z_i[t]+u_i[t-\tau_i]+D_{i}[t+\sigma_i]\big) - a_i\delta_{i-1}[t]\\
+ c_i\mu_i[t] + d_i[t] -D_{i}[t+\sigma_i].
\end{multline}

The parameters in these control laws (the symbols without a time index, which includes $X_i(1)$) are calculated in a simple and structured manner by \cref{alg:init}. Of course having an efficient method for computing the control law is less critical than having an efficient real time implementation of the control law (which is performed by \cref{alg:com_gen}), since the control law can be computed ahead of time. However the fact that this step is also highly structured indicates that the approach is scalable, since it allows for the the control law to be simply and efficiently updated in response to changes to the dynamics in \eqref{eq:dynamics} (perhaps resulting from the introduction of more nodes).

\cref{alg:init} computes all the parameters needed to give  a closed form solution for the problem in \eqref{eq:problem}.
The origin of the parameters in \cref{alg:init} are discussed briefly in the proof idea in \cref{sec:proof_idea} and full details are found in the proof in the appendix.
The algorithm consists of three serial sweeps. The first sweep starts at node 1 and calculates  $\gamma_i$ and  $\rho_i$. The second sweep starts at node $N$, and calculates $X_i(t)$ 
The calculation of $X_i(t)$ has both local steps (line 10) and steps that requires communication (line 9). Also during the second sweep, the parameters $g$, $b$ and $P$ are calculated locally. The third sweep starts at node 1 again, and calculates the parameter $h$, which is needed to calculate the optimal production. Finally, after the third sweep, the parameters $\phi_i(\Delta)$, $a_i$ and $c_i$ are calculated in each node independently.  

The real time implementation of the optimal control law also has a simple distributed implementation. This is the role of \cref{alg:com_gen}, and the structure of the implementation is illustrated in \cref{fig:com_ill}. The algorithm proceeds through two sweeps through the graph. These sweeps are independent of one another, and can be conducted in parallel. In the upstream sweep (from node 1 to node $N$), a set of local variables ($\Phi_i[t]$ and $\delta_i[t]$) are computed according to lines 2--3. This is done sequentially, since the computation of $\delta_i[t]$ depends on $\delta_{i-1}[t]$. $\delta_{i-1}[t]$ then gives all the information node $i$ needs from downstream nodes. Similarly the downstream sweep sequentially computes the $\pi_i[t]$ and $\mu_i[t]$ variables. Here $\mu_i[t]$ gives all the information needed from nodes upstream of node $i$. Once these two sweeps are completed, the optimal inputs can be calculated locally using lines 11 and 12.

\subsection{Receding Horizon and Calculation of $D_i[t]$}\label{sec:recedinghorizon}

We will now discuss how to implement the controller in a receding horizon style to account for updates and new information about the planned disturbances $d_i[t]$. In terms of both the optimal control problem in \eqref{eq:problem} and the controller implementation in \cref{alg:com_gen}, the planned disturbances are treated as fixed quantities, that are known up to some horizon length $H$ into the future (and equal to zero thereafter, c.f. \cref{assump:horizon}). The idea is that $d_i[t]$ determines the anticipated consumption of the quantity at node $i$ and time $t$. Having this information available ahead of time allows the optimal control law to anticipate the predicated usage, and optimally 'schedule' the transportation of the quantity through the network. As we will see in the examples this can lead to a significant improvement in performance. However, in practice we would want to update the values of the $d_i[t]$'s as time passes, and more up to date information becomes available. 

A natural way to do this is to use a receding horizon approach. In this setting we assume that at each point in time, we essentially have a fresh problem, with a new set of planned disturbances. \cref{alg:com_gen} can then be used to compute the first optimal input for this problem, after which the problem resets, and we get a new horizon of planned disturbances. This ensures we always make the best action available to us with a given horizon of information about the disturbances. The question is then, how to efficiently update the part of the control law that depends on the planned disturbances. 

The changes required to accommodate this are rather minor. The planned disturbances do not affect the control parameters or the distributed structure of the implementation of the control law. To see this, observe from \cref{alg:init,alg:com_gen} that all of the information about the planned disturbances is handled through the variables $D_i[t]$ defined in \cref{thm:gen}. 
An illustration of the relationship between individual disturbance $d_i$ and shifted disturbance vectors $D_i[t]$ can be found in \cref{fig:D_illustration}. 
\begin{figure}
    \centering
    \begin{tikzpicture}[->,shorten >=1pt,auto,node distance=1.0cm,
                    thick,node/.style={circle,draw,radius = 1cm}]
        \def\dx{2};
        \def\dy{0.75};
        \def\cx{0.35};
        \def\cy{0.35};
        \def\cxp{0.25}
        \def\cyp{0.25}
        \def\dd{1.25};
        \def\ddist{0.75};
        \def\rotang{-45};
        \tikzstyle{dot} = [circle,draw,inner sep=1.5pt,fill]
        \node[dot] at(0,0)  (n3-0) {};
        \node[dot] at(\dx,0)  (n2-0) {};
        \node[dot] at(2*\dx,0)  (n1-0) {};
        \node[dot] at(0,-\dy)  (n3-1) {};
        \node[dot] at(\dx,-\dy)  (n2-1) {};
        \node[dot] at(2*\dx,-\dy)  (n1-1) {};
        \node[dot] at(0,-2*\dy)  (n3-2) {};
        \node[dot] at(\dx,-2*\dy)  (n2-2) {};
        \node[dot] at(2*\dx,-2*\dy)  (n1-2) {};
        \node[dot] at(0,-3*\dy)  (n3-3) {};
        \node[dot] at(\dx,-3*\dy)  (n2-3) {};
        \node[dot] at(2*\dx,-3*\dy)  (n1-3) {};
        \node[dot] at(0,-4*\dy)  (n3-4) {};
        \node[dot] at(\dx,-4*\dy)  (n2-4) {};
        \node[dot] at(2*\dx,-4*\dy)  (n1-4) {};
        \node[dot] at(0,-5*\dy)  (n3-5) {};
        \node[dot] at(\dx,-5*\dy)  (n2-5) {};
        \node[dot] at(2*\dx,-5*\dy)  (n1-5) {};
        \node[dot] at(0,-6*\dy)  (n3-6) {};
        \node[dot] at(\dx,-6*\dy)  (n2-6) {};
        \node[dot] at(2*\dx,-6*\dy)  (n1-6) {};
        
        \draw[-,dashed] (2*\dx - \cx,-\dy)--(2*\dx,-\dy+\cy)--(2*\dx + \cx,-\dy)--
        (2*\dx,-\dy-\cy) --  (2*\dx - \cx,-\dy);
        \draw[-,dashed] (2*\dx - \cx,-2*\dy)--(2*\dx,-2*\dy+\cy)--(2*\dx + \cx,-2*\dy)--
        (2*\dx,-2*\dy-\cy) --  (2*\dx - \cx,-2*\dy);
        \draw[-,dashed] (\dx - \cx,-\dy)--(\dx,-\dy+\cy)--(2*\dx + \cx,-3*\dy)--
        (2*\dx,-3*\dy-\cy) --  (\dx - \cx,-\dy);
        \draw[-,dashed] (\dx - \cx,-2*\dy)--(\dx,-2*\dy+\cy)--(2*\dx + \cx,-4*\dy)--
        (2*\dx,-4*\dy-\cy) --  (\dx - \cx,-2*\dy);
        \draw[-,dashed] ( - \cx,-1*\dy)--(0,-1*\dy+\cy)--(2*\dx + \cx,-5*\dy)--
        (2*\dx,-5*\dy-\cy) --  ( - \cx,-1*\dy);
        \draw[-,dashed] ( - \cx,-2*\dy)--(0,-2*\dy+\cy)--(2*\dx + \cx,-6*\dy)--
        (2*\dx,-6*\dy-\cy) --  ( - \cx,-2*\dy);
        
        \node at(2*\dx+\dd,-\dy) {$D_1[t]$};
        \node at(2*\dx+\dd,-2*\dy) {$D_1[t+1]$};
        \node at(2*\dx+\dd,-3*\dy) {$D_2[t+2]$};
        \node at(2*\dx+\dd,-4*\dy) {$D_2[t+3]$};
        \node at(2*\dx+\dd,-5*\dy) {$D_3[t+4]$};
        \node at(2*\dx+\dd,-6*\dy) {$D_3[t+5]$};
        
        \node at (-\ddist,0) (d1-0){};
        \draw[ultra thick] (d1-0) -- (n3-1);
        \node[rotate=\rotang] at (-\ddist-\cxp,\cyp) {$d_3[t]$};
        \node at (-\ddist,-\dy) (d1-1){};
        \draw[ultra thick] (d1-1) --  (n3-2);
        \node[rotate=\rotang] at (-\ddist-1.5*\cxp,-\dy+1.5*\cyp) {$d_3[t+1]$};
        \node at (-\ddist+\dx,0) (d2-0){};
        \draw[ultra thick] (d2-0) --  (n2-1);
        \node[rotate=\rotang] at (-\ddist+\dx-\cxp,0+\cyp) {$d_2[t]$};
        \node at (-\ddist+\dx,-\dy) (d2-1){};
        \draw[ultra thick] (d2-1) --  (n2-2);
        \node[rotate=\rotang] at (-\ddist+\dx-1.5*\cxp,-\dy+1.5*\cyp) {$d_2[t+1]$};
        \node at (-\ddist+2*\dx,0) (d1-0){};
        \draw[ultra thick] (d1-0) --  (n1-1);
        \node[rotate=\rotang] at (-\ddist+2*\dx-\cxp,0+\cyp) {$d_1[t]$};
        \node at (-\ddist+2*\dx,-\dy) (d1-1){};
        \draw[ultra thick] (d1-1) --  (n1-2);
        \node[rotate=\rotang] at (-\ddist+2*\dx-1.5*\cxp,-\dy+1.5*\cyp) {$d_1[t+1]$};
        
    \end{tikzpicture}
    \caption{Illustration for the terms included in $D_i[t]$ for a three node graph with $\tau_1 = \tau_2 = 2$. The first row of dots corresponds to $z_3[t],\ z_2[t],\ z_1[t]$ and the second corresponds to $z_3[t+1],\ z_2[t+1],\ z_1[t+1]$, and so on. Due to lack of space only $d_i[t]$ and $d_i[t+1]$ has been drawn. However, the pattern follows through the graph. From the figure we can see that $D_1[t+3] = D_2[t+3]-d_2[t]$, corresponding to the first part of \cref{alg:com_D}.}
    \label{fig:D_illustration}
\end{figure}
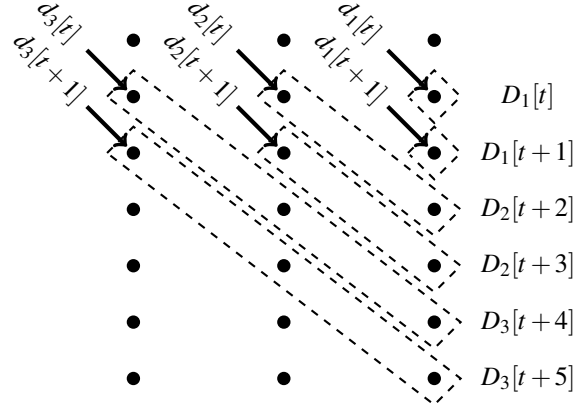
Thus the structure of the implementation of the control law remains the same, and only  $D_i[t]$ needs to be updated as new information become available. This can be done efficiently by a sweep starting at the bottom of the graph. After all, although in the receding horizon framework we assume we have a `new' set of planned disturbances at each point in time, these will share a large amount of information with the planned disturbances from the previous time step. This is the role of \cref{alg:com_D}, which we now explain.
\SetInd{0.5em}{1em}
\begin{algorithm}[]
\SetNoFillComment
\KwIn{changed $d_i[s]$}
\KwOut{updated $D_i$}
\vspace{-0.25cm}
\nonl \hrulefill  \\
\vspace{-0.15cm}
\tcc{Update Disturbances in parallel, $\mathcal{O}(1)$. Necessary even if there are no new disturbances}
\For{node i = N:1}{
\textbf{Send} $D_{i-1}[t+1+\sigma_{i}-1] = D_{i}[t+\sigma_{i}] - d_{i}[t]$ downstream.\\
\textbf{Discard} $D_i[t+\sigma_{i}]$ \\
}
\vspace{-0.25cm}
\nonl \dotfill\\
\vspace{-0.15cm}
\tcc{Update Disturbances due to new planned disturbances}
\For{node i = 1:N} {
\tcc{ For $t+\sigma_i\leq s <t+\sigma_N+H$}
    \If{$d_i[s-\sigma_i]$ \emph{changed} or $D_{i-1}[s]$ \emph{received}}
    {\textbf{send} $D_i[s] = D_{i-1}[s] + d_i[s-\sigma_i] $ upstream
    }    
}
\caption{Calculation of $D$}
\label{alg:com_D}
\end{algorithm}

All the $D_i[t]$'s where none of the underlying $d_j[t]$ were changed can easily be updated. For $1\leq \Delta \leq \tau_i-1$ the information in $D_i[t + \sigma_i + \Delta]$ will be useful in node $i$ at the next time step as  $$D_i[t+1+\sigma_i + \Delta-1] = D_i[t + \sigma_i + \Delta].$$
When $\Delta = 0$ the information can be used at the downstream neighbor as $D_i$ satisfies
\[
D_{i-1}[t+1+\sigma_{i}-1] = D_{i}[t+\sigma_{i}] - d_{i}[t].
\] 
These updates can be done in all nodes simultaneously and the time it takes is thus independent of the size of the graph. 

However,  the shifted sums $D_i$ has to be initialized at time zero, and also updated when new disturbances $d_i$ are planned for times $t>0$.
$D_i[t]$ only requires information from downstream, and can thus be calculated by a sweep starting at node one and going upstream.
Starting at the first node, any of the $d_1[s],\ t\leq s\leq t+\sigma_N+H$ that have changed are sent to node two. Then for every node $i$, the aggregate 
$D_i[s],\ t+\sigma_{i} \leq s\leq t+\sigma_N+H$
is sent upstream if it has changed.
This will be the case if node $i$ received
$D_{i-1}[s],\ t+\sigma_{i}\leq s \leq t+\sigma_N+H$ from its downstream neighbor, or if $d_i[s],\ t\leq s \leq \sigma_N-\sigma_i+H$ has changed.

 The steps for updating the disturbances are summarized in \cref{alg:com_D}.
While the algorithm is essentially a sweep through the graph in the upstream direction, it might be best to not implement it in the upstream sweep of \cref{alg:com_gen} as then the downstream sweep would have to be done after the upstream sweep, due to its need for the shifted disturbance vectors.
 On the other hand, the calculation does not rely on measurement from the system, and can thus be carried out either before or after \cref{alg:com_gen}.

We are now ready to discuss how the choice of $H$ affect the implementation of the controller. Firstly, a larger $H$ will lead to a very slight increase in the synthesis time due to more iterations of $X_N(t)$ being required. Secondly increasing $H$ will increase the memory requirement in node $N$, in that it requires to store $D_N[t+\Delta]$ for $0\leq \Delta \leq H$. Finally the requirement for the communication bandwidth when updating $D_i[t]$ will depend on the number of new disturbances $d_j[t]$, but is upper bounded by $H$ if $d_i[t] = 0$ for  $t>H$ and by $H+\sigma_N$ if $d_i[t] = 0$ for $t>H+\sigma_N-\sigma_i$. Thus if the bandwidth is limited, and a lot of new disturbances are expected to be planned, one might need to limit the size of $H$. Otherwise it can be freely chosen based on the nodes abilities to forecast disturbances.

\section{Simulations}\label{sec:simulations}

In this section we  explore the effect the feed-forward  horizon has on the controller performance through simulations.
In \cref{fig:cost_vs_horizon} the performance for different horizon lengths is shown. Two random nodes are affected by disturbances  of total size between minus one and zero and during a time interval of length between 1 and 5. The node level cost is given by $q_i = 1$. The production cost is given by $r_i = 10N$, where $N$ is the number of nodes. This is an attempt to keep the production cost similar for different values of $N$. There are 50 simulations done for each case with random disturbances as previously described. For all the cases when $N$ is the same, all the disturbances are the same for all the different horizons and delay values.
The horizon lengths are the same for all nodes, i.e it is assumed that $d_i[t+d] = 0$ for $d>H$.

We can see that a large part of the performance increase form having feed-forward can be achieved for short disturbance horizons. We can also see that for larger delays, and for more nodes, a longer horizon is needed to get the same effect. As a rule of thumb, at least for this example, it seems like a horizon longer than 2/3 of the total delay gives almost no effect, and even a horizon of 1/3 of the total delay gives most of the performance increase.

\begin{figure}[]
  \centering
  \includegraphics{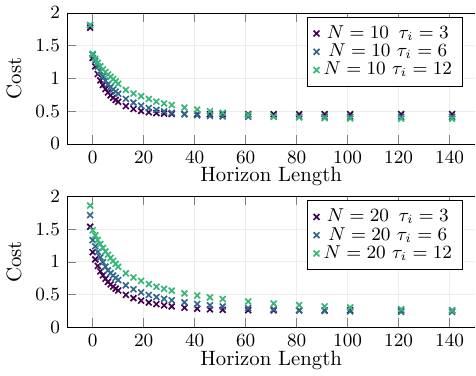}
  \caption{Simulations comparing the effect the planning horizon has on the performance. The data-points with highest cost for each configuration corresponds to not using the planned disturbances at all. While the the rest corresponds to using all planned disturbances announced up to $H$ time units ahead in every node.}
  \label{fig:cost_vs_horizon}
\end{figure}

\section{Proof Idea}\label{sec:proof_idea}
In this section we will describe the main idea behind the technique used to derive the results, which is to study a time shifted sum of the node-levels $z_i$, and a time shifted sum of the production $v_i$. This will allow the problem to be solved in terms of these shifted sums, essentially reducing it to a problem with scalar variables. Outside the disturbance horizon the problem can be solved by a Riccati equation in one variable. While inside the disturbance horizon the problem is solved using dynamic programming, where each step has scalar variables.

Now for the definitions of the shifted sums, let the sum of a shifted level $S_k$ and sum of a shifted production vector $V_k$ be defined as
\[
  S_k[t] = \sum_{i=1}^k z_i[t-\sigma_i], \qquad V_k[t] = \sum_{i=1}^k v_i[t-\sigma_i].
\] 
Also Let $\bar{V}_k[t] = V_k[t]+D_k[t]$ to shorten some expressions.

We illustrate the main idea by considering these shifted sums with a short example. Consider a path graph with $N>2$, $\tau_1 = 1$, and $\tau_2>1$. 
Then  $S_2[3] =  z_1[3]+z_2[2]$. It can be checked that  
\[S_2[3] = z_1[0]+z_2[0] + \bar V_1[0]+\bar V_2[1]+\bar V_2[2] + u_1[-1] + u_2[-\tau_2].\]
Note that the internal transportation $u_1[0]$ and $u_1[1]$ have canceled, and the sum is thus independent of the internal transportation $u$ (except those with negative time index, which correspond to initial conditions).  Also note that any values for $z_2[2]$ and $z_1[3]$ can be achieved as long as $z_1[3]+z_2[2] = S_2[2]$. This follows from that $z_2[2]$ can take any value by choosing the appropriate value for $u_1[1]$. Thus the cost of $q_2z_2[2]^2 + q_1z_1[3]^2$ only depends on the value of $S_2[3]$, which is independent of all internal transportation $u_i[t],\ t\geq 0$. 
This means that all inputs except $u_1[1]$ can ignore its effect on the terms in $S_2[3]$. Furthermore,  $S_2[3]$, and thus the  corresponding cost, only depends on the sums $V_2[1]$ and $V_2[2]$ and not the individual productions $v_1[1]$, $v_1[2]$, $v_2[0]$ and $v_2[1]$.

This idea can be generalized. The cost function can be rewritten in terms of shifted levels, where each shifted level sum is independent of the internal transportation. Each shifted level can thus be minimized independently with respect to the internal transportation $u$. Just as in the example, the only constraint is that the shifted sum has the correct value. The optimal cost for a shifted vector $S_k[t]$ is given by the solution to
    \begin{equation}\label{eq:opt_zi}
    \begin{aligned}
      \minimize_{z_i}  \quad & \sum_{i=1}^k q_iz_i[t+\sigma_k-\sigma_i]^2 \\
     \textrm{subject to} \quad & S_k[t+\sigma_k] = c,
    \end{aligned}\end{equation}
where $c$ depends on the initial conditions, $V_i$, and $D_i$.
The problem has the solution $z_i[t+\sigma_k-\sigma_i] = \gamma_k/q_i\cdot c$ and cost $\gamma_kc^2$, where $\gamma_k$ is as given in \cref{alg:init}. Once the optimal level $z_k[1]$ is calculated,  the optimal value for $u_{k-1}[0]$ can be found from the dynamics, which gives
\begin{multline*}
  u_{k-1}[0] = (1-\frac{\gamma_{k}}{q_k})(z_{k}[0] + u_{k}[-\tau_{k}]) + v_k[0] +d_k[0]\\
  -\frac{\gamma_k}{q_k}\bar{V}_k[\sigma_k]
  -\frac{\gamma_k}{q_k}(m_{k-1}[0] + \sum_{i=1}^{k-1} \sum_{d=0}^{\tau_i-1}\bar{V}_i[\sigma_i+d]).
\end{multline*}
Where $m_k[t] = \sum_{i=1}^k(z_i[t] +  \sum_{\delta = 1}^{\tau_i}u_i[t-\delta])$. After inserting the optimal values for $v_k$ and $V_i$ the expression in \eqref{eq:opt_u} is achieved. Note that all the terms with coefficient $1$ corresponds to what would be in the node $k$ at time $t=1$  if $u_{k-1}[0] = 0$, and all terms with coefficient $-\gamma_k/q_k$  gives the total quantity in $S_k[1]$. 

Furthermore, each shifted level sum only depends on the shifted production sums $V_k[t]$, and not the individual productions $v_i[t],\ i\leq k$. The optimal way to produce a specific amount $V_k[t]$ with a shifted production vector can be found by solving a problem similar to \eqref{eq:opt_zi}, with the optimal $v_i$ given by $v_i[t-\sigma_i] = \rho_k/r_i\cdot V_k[t] \text{ for } i \leq k$, and the cost given by $\rho_kV_k[t]^2$, where $\rho_k$ is as given in \cref{alg:init}. So the calculations of $\gamma_i$ and $\rho_i$ in the first sweep in \cref{alg:init} thus corresponds to solving the optimal distribution for a shifted level vector $S_k$ and the optimal production for a shifted production vector $V_k$. This is covered in \cref{lem:shifted_inv}.

Assuming all $u_i$ are picked so that the shifted levels are optimized,  the total level cost is given by
  \begin{equation}\label{eq:cost_S}
  \sum_{t=0}^\infty\sum_{i=1}^N q_i z_i[t]^2 = \sum_{i=0}^N q_i z_i[0]^2
  + \sum_{i=1}^{N-1}\ \ \ \  \mathclap{\sum_{t=\sigma_i+1}^{\sigma_{i+1}}}\ \ \gamma_i S_i[t]^2 + \ \ \mathclap{\sum_{t=\sigma_N+1}^\infty}\ \   \gamma_NS_N[t]^2.
\end{equation}
And assuming all $v_i$ are picked so that each shifted production vector is optimized, then the total production cost is given by
\begin{equation}\label{eq:cost_V}
  \sum_{t=0}^\infty\sum_{i=1}^N r_iv_i[t]^2 =   \sum_{i=1}^{N-1}\sum_{t=\sigma_i}^{\sigma_{i+1}-1}\rho_i V_i[t]^2 + \sum_{t=\sigma_N}^\infty \rho_NV_N[t]^2.
\end{equation}
This allows the problem in \eqref{eq:problem} to be solved in terms of $V_i$ and $S_i$, reducing it to a problem in scalar variables. The scalar problem can be solved analytically, giving a closed form solution.

This is done by first solving for all $V_N[t]$ outside the disturbance horizon, that is for $t >\sigma_N + H$. Using that outside the disturbance horizon the dynamics for the shifted levels $S_N[t]$ are $S_N[t+1] = S_N[t] +V_N[t]$ gives that those $V_N[t]$ are given by the solution to
 \[\begin{aligned}
  \minimize_{V_N[t]} \quad & \sum_{\mathclap{t=\sigma_N+H+1}}^\infty \gamma_NS_N[t]^2 + \rho_NV_N[t]^2 \\
  \text{subject to} \quad & S_N[t+1] = S_N[t] +V_N[t].
\end{aligned}\]
This problem can be solved through a Riccati equation in one variable, giving expressions for  $V_N[t],\ t >\sigma_N + H$ in terms of $S_N[t]$. And more importantly, a cost to go in terms of $S_N[\sigma_N+H+1]$, that is $X_N[H+2]$ in \cref{alg:init}.

Each shifted sum in \eqref{eq:cost_S} can be expressed in terms of initial conditions, shifted production vectors, and shifted disturbance vectors,
\begin{equation*}
  S_k[\sigma_k+\Delta] = m_{k-1}[0]+ z_k[0] + \sum_{i=1}^{k-1}\sum_{d=\sigma_i}^{\sigma_{i+1}-1}\bar{V}_i[d] 
  + \sum_{d = \sigma_k}^{\mathclap{\sigma_k+\Delta-1}} \bar{V}_k[d] .
\end{equation*}
 Using the cost to go from the Riccati equation as the terminal cost allows the rest of the $V_i$'s to be found analytically using dynamic programming. When solving this problem the cost to go $X_i$ in \cref{alg:init} is used. The parameter $g$ also appears naturally in the solution to each dynamic programming step, and the upstream aggregate $\mu_i$ in \cref{alg:com_gen} is used to simplify the expressions.  This is covered in detail in \cref{lem:V_expr}.

  The resulting solution gives $V_i[t]$  in terms of initial conditions and the previous $V_k$'s in  \eqref{eq:cost_V}.
However, $V_1[0]$ is known, which gives $V_1[1]$ and so on. When rewriting  $V_i[0]$ in terms of only initial conditions the expressions can be simplified by using $\delta$ as defined in \cref{alg:com_gen}. This in turn requires $P$, $h$, and $\phi$ which were defined in \cref{alg:init} and $\Phi$ which was defined in \cref{alg:com_gen}. For the details see \cref{lem:mVsum}.

\section{Conclusions and Future Work}
In this paper we studied an optimal control problem on a simple transportation model. We showed that the optimal controller is highly structured, allowing for a distributed implementation consisting of two sweeps through the graph. The optimal controller can also handle planned disturbances in an efficient way.

We believe that the results presented here can be extended to more general graph structures. More specifically for any graph with the structure of a directed tree both the proof technique and the results could be extended. We plan to explore this in a future publication.



\appendix
\section{Appendix}

The proof follows the structure of the proof idea. Before we start we restate the definition of $m_k$ which was mentioned in the proof idea.
\[ m_k[t] = \sum_{i=1}^k\left( z_i[t] + \sum_{d=1}^{\tau_i} u_i[t-d]\right)\]
Also, we let the product over an empty set be equal to one, e.g., $\prod_{i=2}^1 g_i = 1$.

The proof will derive the optimal inputs at time $t = 0$. As the problem has an infinite horizon, one can freely shift the time, and the results will thus holds for all $t \geq 0$.
We begin by showing that each shifted level can be optimally distributed and find the corresponding internal flows.
\begin{lemma}\label{lem:shifted_inv}
The following holds
\begin{enumerate}[(i)]
  \item Every shifted level $S_k$ satisfies
  \begin{multline*}S_k[t+\sigma_k+1] =\\ z_k[t] + u_k[t-\tau_k] +m_{k-1}[t]
   + \sum_{i=1}^{k-1} \sum_{d=0}^{\tau_i-1}\bar{V}_i[t+\sigma_i+d]
   +\bar{V}_k[t+\sigma_k]
   \end{multline*}
  \item Let $\gamma_k$ be defined as in \cref{alg:init}. The optimization problem
    \[\begin{aligned}
      \minimize_{z_i}  \quad & \sum_{i=1}^k q_iz_i[t+\sigma_k-\sigma_i]^2 \\
     \textrm{subject to} \quad & S_k[t+\sigma_k] = m,
    \end{aligned}\]
    has the solution $z_i = \gamma_k/q_im$ and the optimum value is given by $\gamma_km^2$.
  \item When $u$ is chosen optimally, the cost for \eqref{eq:problem} is given by
  \begin{equation*}
  \sum_{t=0}^\infty\sum_{i=1}^N q_iz_i[t]^2 = \sum_{i=0}^Nq_iz_i[0]^2
  + \sum_{i=1}^{N-1}\sum_{t=\sigma_i+1}^{\sigma_{i+1}}\gamma_iS_i[t]^2 + \sum_{t=\sigma_N+1}^\infty \gamma_NS_N[t]^2.
\end{equation*}
  Also, the optimal $u_k[0]$  is given by
\begin{multline*}
  u_{k-1}[0] = (1-\frac{\gamma_{k}}{q_k})(z_{k}[0] + u_{k}[-\tau_{k}]) + v_k[0] +d_k[0]\\
  -\frac{\gamma_k}{q_k}\bar{V}_k[\sigma_k]
  -\frac{\gamma_k}{q_k}(m_{k-1}[0] + \sum_{i=1}^{k-1} \sum_{d=0}^{\tau_i-1}\bar{V}_i[\sigma_i+d]).
\end{multline*}
   
\end{enumerate}
\end{lemma}
\begin{proof}
For $k=1$ (i) reduces to the dynamics. Now assume that (i) holds for $k-1$. It follows from the definition of $S_k$ that 
\begin{equation}\label{eq:S_ind_shift}
S_k[t+\sigma_k+1] = z_k[t+1] + S_{k-1}[t+\sigma_k+1].
\end{equation}
It holds that
 \begin{equation}\label{eq:Sk_step}
 S_k[t+1] = S_k[t]+\bar{V}_k[t] + u_k[t-\sigma_k-\tau_k],
 \end{equation}
 since $u_i[t-\sigma_i-\tau_i]$ will cancel out for $i<k$.
  This allows $S_{k-1}[t+\sigma_k+1]$ to be rewritten as
  \begin{equation}\label{eq:sk_big_step}
    S_{k-1}[t+\sigma_k+1] =  S_{k-1}[t+\sigma_{k-1}+1]+\ \ \mathclap{\sum_{\Delta=\sigma_{k-1}+1}^{\sigma_k}}\ \ \ \bar{V}_{k-1}[t+\Delta] + \sum_{\Delta =0}^{\tau_{k-1}-1} u_{k-1}[t-\Delta].
  \end{equation}
  Using the induction assumption that (i) holds for $k-1$, \eqref{eq:sk_big_step} and the dynamics,
    \begin{equation}\label{eq:dyn_in_proof}
    z_k[t+1] = z_k[t] -u_{k-1}[t]+ u_k[t-\tau_k]+v_k[t]+d_k[t],
  \end{equation} 
  allows \eqref{eq:S_ind_shift} to be rewritten as
  \begin{multline*}
   S_k[t+\sigma_k+1] = z_k[t] - u_{k-1}[t] + u_k[t-\tau_k]+v_k[t]+d_k[t] \\
   +z_{k-1}[t] + u_{k-1}[t-\tau_{k-1}] +m_{k-2}[t] 
   + \sum_{i=1}^{k-2} \sum_{d=0}^{\tau_i-1}\bar{V}_i[t+\sigma_i+d]\\
   +\bar{V}_{k-1}[t+\sigma_{k-1}] 
   +\ \ \mathclap{\sum_{\Delta=\sigma_{k-1}+1}^{\sigma_k}}\ \ \  \bar{V}_{k-1}[t+\Delta] + \sum_{\Delta =0}^{\tau_k-1} u_{k-1}[t-\Delta].
  \end{multline*}
In the above it holds that
\[
z_{k-1}[t]+u_{k-1}[t-\tau_{k-1}] -u_{k-1}[t]+\sum_{\Delta =0}^{\mathclap{\tau_{k-1}-1}} u_{k-1}[t-\Delta] + m_{k-2}[t] = m_{k-1}[t]
\]
  and  
  \begin{multline*}
    v_k[t]+d_k[t] + \sum_{i=1}^{k-2} \sum_{d=0}^{\tau_i-1}\bar{V}_i[t+\sigma_i+d] + \bar{V}_{k-1}[t+\sigma_{k-1}] + \ \ \mathclap{\sum_{\Delta=\sigma_{k-1}+1}^{\sigma_k}}\ \ \ \bar{V}_{k-1}[t+\Delta]\\
     = \sum_{i=1}^{k-1} \sum_{d=0}^{\tau_i-1}\bar{V}_i[t+\sigma_i+d]
   +\bar{V}_k[t+\sigma_k].
  \end{multline*}
  And thus (i) holds for $k$ as well.

For (ii) the proposed solution satisfies the constraint as
\[ 
    \sum_{i=1}^k \frac{1}{q_i} = \frac1{\gamma_k}.
\]
If the proposed solution was not optimal then it would be possible to improve it by increasing $z_i$ by epsilon and decreasing $z_j$ by epsilon for $i,j\leq K$ as the problem is convex. However
\[
    \frac{\partial}{\partial z_i} q_iz_i[t+\sigma_k-\sigma_i]^2 = 2\gamma_km
\]
for $z_i[t+\sigma_k-\sigma_i] = \gamma_k/q_i m$ and all $i$, and thus the proposed solution is optimal.

  For (iii)   note that the sum $\sum_{t=0}^\infty\sum_{i=1}^N q_iz_i[t]^2$ can be written in terms of shifted level vectors as follows,
\begin{multline}
\label{eq:cost_decomp}
  \sum_{t=0}^\infty\sum_{i=1}^N q_iz_i[t]^2 = \\
   \sum_{i=0}^Nq_iz_i[0]^2
  + \sum_{i=1}^{N-1}\sum_{t=\sigma_i+1}^{\sigma_{i+1}} \sum_{j=1}^iq_jz_j[t-\sigma_j]^2 + \ \ \mathclap{\sum_{t=\sigma_N+1}^\infty}\ \ \ \ \  \sum_{j=1}^N q_jz_j[t-\sigma_j]^2.
\end{multline}
The inner sums corresponds to the objective in (ii). From (i) it follows that $S_k[t]$, $t\leq \sigma_{i+1}$ is independent of $u_j[t], \ \forall t \geq 0, \forall j$ and that $S_N[t]$ is independent of $u_j[t],\ \forall t,j$. Thus each shifted level sum in \eqref{eq:cost_decomp} is independent of the internal flows. Now consider arbitrary, but fixed productions $V$ and disturbances $D$. Then by (i) the sum of all shifted levels are fixed. If there exists $u$ so that each sum over shifted levels in \eqref{eq:cost_decomp} is the optimal solution to the problem in (ii), then those inputs must be optimal for the given $V$ and $D$. By choosing $u_{j-1}[t-\sigma_j-1]$ so that $z_j[t-\sigma_j]$  is optimal for (ii) for $2\leq j \leq i$ gives that all $z_j[t-\sigma_j]$ are optimally for $2\leq j \leq i$. However, since the constraint will always be satisfied, $z_1[t]$ will be optimal as well.
Using (i), the optimal $z_k[1]$ from (ii) is given by
\begin{equation*}
  z_k[1] = \frac{\gamma_k}{q_k}\Big(m_{k-1}[0] +z_k[0] + u_k[-\tau_k]  
  +  \sum_{i=1}^{k-1}\sum_{\Delta = 0}^{\tau_i-1}\bar{V}_i[\sigma_i +\Delta] + \bar{V}_k[\sigma_k] \Big).
\end{equation*}
Inserting the dynamics in \eqref{eq:dyn_in_proof} into the LHS and solving for $u_{k-1}[0]$ gives the expression in (iii).
\end{proof}

Now we will give the solution to the optimization problem which will arise in the dynamic programming problem that will need to be solved in the next lemma.
\begin{lemma}\label{lem:aux}
Let $X_i$ and $g_i(j)$ be defined as in \cref{alg:init}. Then
\begin{enumerate}[(i)]
 
\item Let $j\geq1$. The optimization problem
\[
  \minimize_x \quad X_i(j+1)(a+b + x)^2 + \gamma_i(a+x)^2 + \rho_ix^2
\]
has minimizer 
\[x = -\frac{X_i(j)}{\rho_i}(a+g_i(j+1)b),\]
 with optimum value $X_i(j)\cdot(a+g_i(j+1)b)^2 + f(b)$.
\item The optimization problem
\[
  \minimize_x \quad X_{i+1}(1)(a+b+ x)^2 + \gamma_{i}(a+x)^2 + \rho_{i}x^2
\]
has minimizer 
\[
  x = -\frac{X_{i}(\tau_{i})}{\rho_{i}}(a+g_{i+1}(1)b),
\]
 with optimum value $X_{i}(\tau_{i})\cdot(a+g_{i+1}(1)b)^2+f(b)$.
\end{enumerate}
\end{lemma}
\begin{proof}
We will show that the optimization problem
\[
  \minimize_x \quad  c_1(a+b+x)^2 + c_2(a+x)^2+c_3x^2
\]
has the solution
\[
  x = -\frac{c_1+c_2}{c_1+c_2+c_3}\big(a + \frac{c_1}{c_1+c_2}b\big)
\]
  and the minimal value is on the form
\[
  \frac{c_3(c_1+c_2)}{c_1+c_2+c_3}\left(a + \frac{c_1}{c_1+c_2}b\right)^2  + f(b),
\]
where $f(b)$ is independent of $a$.
The lemma then follows by applying the above and using the definition for $X_i$ and $g_i(j)$.

  There exits a unique solution as the problem is strictly convex.
  Differentiating the objective function with respect to $x$ gives that the optimal $x$ is given by
\[
  x = -\frac{1}{c_1+c_2+c_3}\big((c_1+c_2)a + c_1b\big)
\]
from which the proposed $x$ follows.
  The objective function can be rewritten as
  \[
    c_1(a+b)^2 + c_2a^2 + 2[(c_1+c_2)a + c_1b)]x + (c_1 + c_2 + c_3)x^2.
  \]
Inserting the minimizer gives
\[
  \frac{1}{c_1+c_2+c_3}\Big((c_1+c_2+c_3)(c_1(a+b)^2+c_2a^2) - [(c_1+c_2)a + c_1b]^2\Big).
\]
The first term can be written as
\begin{multline*}
  (c_1+c_2+c_3)(c_1(a+b)^2+c_2a^2)  \\=(c_1+c_2+c_3)\big[(c_1+c_2)a^2 + 2c_1ab + c_1b^2 \big]\\
  =(c_1+c_2)^2a^2 + 2(c_1+c_2)c_1ab + c_1^2b^2 + \\
  c_3(c_1+c_2)a^2 + 2c_1c_3ab + (c_2+c_3)c_1b^2.
\end{multline*}
Which gives that the objective function has the minimum value
\[
  \frac{1}{c_1+c_2+c_3}\Big[c_3(c_1+c_2)a^2 + 2c_1c_3ab + (c_2+c_3)c_1b^2\Big].
\]
The last term is independent of $a$. The dependence on a is thus given by
\[
  \frac{c_3(c_1+c_2)}{c_1+c_2+c_3}\left(a^2 + \frac{2c_1}{c_1+c_2}ab\right) = \frac{c_3(c_1+c_2)}{c_1+c_2+c_3}\left(a + \frac{c_1}{c_1+c_2}b\right)^2 + f(b)
\]
\end{proof}

Armed with the results from the previous lemma, we will now apply dynamic programming to \eqref{eq:problem}. We will show that the problem can be solved in terms of the shifted levels $S_k$, shifted productions $V_k$ and shifted disturbances $D_k$. Outside of the horizon the problem can be solved using the Riccati equation. Using the cost to go given by the Riccati equation as initialization we can apply dynamic programming using the results from the previous lemma.
\begin{lemma}\label{lem:V_expr}
Let  $\gamma_k$, $\rho_k$, $X_k$, $g_k$, $\mu_k$ be defined as in \cref{alg:init,alg:com_gen}. Let for $1\leq k\leq N-1$ and $1\leq \Delta \leq\tau_k$,  and for $k=N$ and $1 \leq \Delta \leq H+2$ 
\begin{multline} \label{eq:xi}
   \xi_k[\Delta-1] = m_{k-1}[0] + z_k[0] +  \sum_{i=1}^{k-1}\sum_{d = \sigma_i}^{\sigma_{i+1}-1}\bar{V}_i[d] \\
   + \sum_{d=0}^{\tau_k-1} \Big(u_k[-(\tau_k-d)]+D_k[\sigma_{k}+d]\Big)\prod_{j=\Delta+1}^{d+1}g_k(j) 
 \\
    +\sum_{d = \sigma_k}^{\sigma_k + \Delta-2}V_k[d] + \mu_{k+1}[0]g_{k+1}(1)\prod_{j=\Delta+1}^{\tau_k} g_k(j).
\end{multline}
Then the optimal $V_k$ for \eqref{eq:problem} is given by
\begin{equation*}
  V_k[\sigma_k+(\Delta-1)] = -\frac{X_k(\Delta)}{\rho_k}\xi_k[\Delta-1].
\end{equation*}
The optimal individual productions are given by
\[ v_k[\Delta-1] = \frac{\rho_k}{r_k}V_k[\sigma_k+(\Delta-1)].\]
\end{lemma}
\begin{proof}
By \cref{lem:shifted_inv}-(i) and \eqref{eq:Sk_step} each shifted inventory level $S_k[\sigma_k +\Delta]$ with $1\leq \delta \leq \tau_{k}$ satisfies
\begin{equation}\label{eq:Sk_delta_shift}
  S_k[\sigma_k+\Delta] = m_{k-1}[0]+ z_k[0] + \sum_{i=1}^{k-1}\sum_{d=\sigma_i}^{\sigma_{i+1}-1}\bar{V}_i[d] 
  + \sum^{\tau_k}_{\mathclap{d = \tau_k-(\Delta-1)}} u_k[-d] 
  + \sum_{\mathclap{d = \sigma_k}}^{\mathclap{\sigma_k+\Delta-1}} \bar{V}_k[d].
\end{equation}
By (iii) in \cref{lem:shifted_inv} the cost can be rewritten as
\begin{equation*}
  \sum_{t=0}^\infty\sum_{i=1}^N q_iz_i[t]^2 = \sum_{i=0}^Nq_iz_i[0]^2 
  + \sum_{i=1}^{N-1}\sum_{t=\sigma_i+1}^{\sigma_{i+1}}\gamma_iS_i[t]^2 + \sum_{\mathclap{t=\sigma_N+1}}^\infty \gamma_NS_N[t]^2.
\end{equation*}
And similarly, by \cref{lem:shifted_inv}-(ii), the optimal cost for a shifted production $V_i[t]$ is given by $\rho_iV_i[t]^2$ and individual productions are given by $v_i[t] = \rho_i/r_i\cdot V_i[t+
\sigma_i]$. This gives the total production cost in terms of $V_i$ as
\[
  \sum_{t=0}^\infty\sum_{i=1}^N r_iv_i[t]^2 =   \sum_{i=1}^{N-1}\sum_{t=\sigma_i}^{\sigma_{i+1}-1}\rho_iV_i[t]^2 + \sum_{t=\sigma_N}^\infty \rho_NV_N[t]^2.
\]

We can thus solve the problem in terms of $S_i$ and $V_i$, and then recover the optimal $v_i$. To that end define the cost to go for $1\leq k \leq N-1$ and $1\leq \Delta \leq \tau_k$
\begin{multline*}
  \Gamma_k[\Delta] = \sum_{t = \sigma_k+\Delta}^{\sigma_{k+1}}\Big(\gamma_kS_k[t]^2+\rho_kV_k[t-1]^2\Big) 
  + \sum_{i=k+1}^{N-1}\sum_{\sigma_i+1}^{\sigma_{i+1}}\Big( \gamma_iS_i[t]^2+\rho_iV_i[t-1]^2\Big)\\
  + \sum_{t=\sigma_N+1}^\infty \Big( \gamma_NS_N[t]^2+\rho_NV_N[t-1]^2\Big).
\end{multline*}
And for $k = N$ and $\Delta \geq 1$
\[
  \Gamma_N[\Delta] =  \sum_{t=\sigma_N+\Delta}^\infty \Big( \gamma_NS_N[t]^2+\rho_NV_N[t-1]^2\Big).
\]
We will show for $1\leq k \leq N-1$ and $1 \leq \Delta \leq \tau_i$, and for $k = N$ and $1\leq \Delta \leq H+2$, that
\begin{equation}\label{eq:cost_to_go}
  \Gamma_k[\Delta] = X_{k}(\Delta)\xi_k[\Delta-1]^2 + f(b), 
\end{equation}
where $f(b)$ is independent of $V_k[t]$. $f(b)$ can thus be ignored in the optimization of $V_k[t]$.

Using \cref{lem:shifted_inv}-(i) combined with \eqref{eq:Sk_step} and that all $D_N[t] =0$ for $t>H+\sigma_N$ it follows that the optimal $V_N[t]$ for $t>\sigma_N+H$ is given by the solution to the problem
\[\begin{aligned}
  \minimize_{V_N[t]} \quad & \sum_{t=\sigma_N+H+1}^\infty \gamma_NS_N[t]^2 + \rho_NV_N[t]^2 \\
  \text{subject to} \quad & S_N[t+1] = S_N[t] +V_N[t] \\
  &S_N[\sigma_N+H+1] = m_N[0] + \sum_{i=1}^{N-1}\ \ \ \ \mathclap{\sum_{\delta=\sigma_i}^{\sigma_{i+1}-1}}\ \ \ \bar{V_i}[\delta] +\ \mathclap{\sum_{\delta = \sigma_{N}}^{\sigma_N+H}}\ \ \ \  \bar{V}_N[\delta].
\end{aligned}\]
This is a standard LQR problem and the solution can be found by solving the following Riccati equation 
\[
  X = X-X^2/(\rho_N+ X)+\gamma_N \Rightarrow X =  \frac{\gamma_N}{2}+\sqrt{\gamma_N\rho_N+\frac{\gamma_N^2}{4}}.
\]
 Now let $X_N(H+2) = X-\gamma_N$. Then  $\Gamma_N[\sigma_N+H+2]$ is given by
 \[\Gamma_N[\sigma_N+H+2] = S_N[\sigma_N+H+1]^2X_N(H+2).\] 
Note that the cost for $S_k[\sigma_N+H+1]$ is not part of $\Gamma_N[\sigma_N+H+2]$, but it is part of the cost to go given by the solution $X$ to the Riccati equation. Furthermore, the optimal $V_N[t]$ for $t = \sigma_N+H+1$ is given by 
\begin{equation*}
  V_N[t] = -\frac{X}{X+\rho_N}S_N[t] 
  = -\frac{X_N(H+1)+\gamma_N}{X_N(H+1)+\gamma_N+\rho_N}S_N[t] = -\frac{X_N(H)}{\rho_N}S_N[t].
\end{equation*}
For $\Delta = H+2$ and $k=N$ the expression for $\xi_k[\Delta-1]$ reduces to
\[
  \xi_N[H+1] = m_N[0] +  \sum_{i=1}^{N-1}\sum_{d = \sigma_i}^{\sigma_{i+1}-1}\bar{V}_i[d] + \sum_{d = \sigma_N}^{\sigma_N + H}\bar{V}_N[d],
\]
as $\mu_{N+1} =  0$, $u_N=0$ and $D_N[t] = 0$ for $t>\sigma_N+H$. By \eqref{eq:Sk_delta_shift} $\xi_N[H+1] = S_N[\sigma_N+H+1]$ and thus the lemma and \eqref{eq:cost_to_go} holds for $k=N$ and $\Delta = H+2$.

Assume that \eqref{eq:cost_to_go} holds for $k+1$ and $\Delta = 1$. Then the optimal $V_{k}[\sigma_{k+1}-1]$ is given by the minimizer for 
\[ \Gamma_{k}[\tau_{k}] = \Gamma_{k+1}[1] + \gamma_{k}S_k[ \sigma_{k+1}]^2 + \rho_k V_k[\sigma_{k+1}-1]^2.\]
Using the assumption for the cost to go in \eqref{eq:cost_to_go} gives that $\Gamma_{k+1}[1] =X_{k+1}(1)\xi_{k+1}[0]^2$ and thus the optimal $V_{k}[\sigma_{k+1}-1]$ is given by the optimal value for the problem
\begin{equation*}
  \minimize_{V_{k}[\sigma_{k+1}-1]}\quad X_{k+1}(1)\xi_{k+1}[0]^2+\gamma_{k}S_{k}[\sigma_{k+1}]^2 
  + \rho_{k}V_{k}[\sigma_{k+1}-1]^2.
\end{equation*}
For $\Delta = 1$  \eqref{eq:xi} reduces to
\begin{equation}\label{eq:delta1reduce}
  \xi_k[0] = m_{k-1}[0] +\sum_{i=1}^{k-1}\sum_{d = \sigma_i}^{\sigma_{i+1}-1}\bar{V}_i[d] +  \mu_k[0],
\end{equation}
as
\begin{equation*}
\mu_k[0]=  \pi_k[0]
 +\mu_{k+1}g_{k+1}(1)\prod_{j=2}^{\tau_k} g_k(j)
\end{equation*}
and
\[
  \pi_k[0] = z_k[0] + \sum_{d=0}^{\tau_k-1} \Big(u_k[-(\tau_k-d)]+D_k[\sigma_{k}+d]\Big)\prod_{j=2}^{d+1}g_k(j).
\]
We also note that by  \eqref{eq:Sk_delta_shift}, as $\sigma_{k+1} = \sigma_k + \tau_k$,
\[ 
S_k[\sigma_{k+1}] = m_k[0] + \sum_{i=1}^{k}\sum_{d=\sigma_i}^{\sigma_{i+1}-1}\bar{V}_i[d].
\]
Applying \cref{lem:aux}-(ii) with 
\[
\begin{aligned}
  a &= S_{k}[\sigma_{k+1}]-V_{k}[\sigma_{k+1}-1]\\
  b &= \xi_{k+1}[0] - S_k[\sigma_{k+1}] = \mu_{k+1}[0] \\
  x &= V_{k}[\sigma_{k+1}-1],
\end{aligned}
\]
 gives that the lemma and \eqref{eq:cost_to_go} hold for $k$ and $\Delta = \tau_{k}$ as
 \[
   \xi_k[\tau_k-1] = m_k[0] + \sum_{i=1}^{k-1}\sum_{d=\sigma_i}^{\sigma_{i+1}-1}\bar V_i[d] + \sum_{d=\sigma_k}^{\mathclap{\sigma_k+\tau_k-2}}\bar{V}_k[d] + \mu_{k+1}[0]g_{k+1}(1).
 \]

Assume that \eqref{eq:cost_to_go} holds for some $k$ and $\Delta+1$, where $1\leq \Delta\leq \tau_i-1$ if $k<N$ and $1\leq \Delta\leq H+1$ if $k = N$. Then $V_k[\sigma_k+\Delta-1]$ can be found as the minimizer for
\begin{equation*} \minimize_{V_k[\sigma_k+\Delta-1]} \quad X_k(\Delta+1)\xi_k[\Delta]^2 + \gamma_k S_k[\sigma_k+\Delta]^2 
  + \rho_k V_k[\sigma_k+\Delta-1]^2.
\end{equation*}
Using that two of the terms in \eqref{eq:Sk_delta_shift} can be rewritten as
\begin{equation*}
  \sum^{\tau_k}_{\mathclap{d = \tau_k-(\Delta-1)}} u_k[-d] 
  + \sum_{d = \sigma_k}^{\mathclap{\sigma_k+\Delta-1}} \bar{V}_k[d] 
  =\sum_{d = 0}^{\Delta-1} \Big(u_k[-(\tau_k-d)]+D_k[\sigma_k+d]\Big) + \sum_{d = \sigma_k}^{\mathclap{\sigma_k+\Delta-1}} V_k[d]
\end{equation*}
 and with $x =V_k[\sigma_k+\Delta-1]$, $a = S_k[\sigma_k+\Delta]-V_k[\sigma_k+\Delta-1]$, which equals
\begin{multline*}
  a = m_{k-1}[0]+ z_k[0] + \sum_{i=1}^{k-1}\sum_{d=\sigma_i}^{\sigma_{i+1}-1}\bar{V}_i[d] \\
  + \sum_{d = 0}^{\Delta-1} \Big(u_k[-(\tau_k-d)]+D_k[+\sigma_k+d]\Big) 
  + \sum_{d = \sigma_k}^{\mathclap{\sigma_k+\Delta-2}} V_k[d],
\end{multline*}
and $b = \xi_k[\Delta]-S_k[\sigma_k+\Delta]$, which
gives
\begin{equation*}
  b = \ \ \mathclap{\sum_{d=\Delta}^{\tau_k-1}} \ \ \ \Big(u_k[-(\tau_k-d)]+D_k[\sigma_{k}+d]\Big)\ \ \ \ \mathclap{\prod_{j=\Delta+2}^{d+1}} \ \ \ g_k(j) 
   + \mu_{k+1}[0]g_{k+1}(1)\prod_{j=\Delta+2}^{\tau_k} g_k(j).
\end{equation*}
By applying \cref{lem:aux}-(i) it follows that \eqref{eq:cost_to_go} and the lemma holds for $k$ and $\Delta$ as well.
 Thus the lemma holds for all $1\leq \Delta \leq\tau_k$ for $1\leq k\leq N-1$ and $1 \leq \Delta \leq H+2$ for $k=N$.
\end{proof}

All that remains now is to find expressions for $V_k[\sigma_k]$ in terms of the initial conditions. The following lemma allows us to do so, using the expressions for $V_k$ derived in the previous lemma.
\begin{lemma}\label{lem:mVsum}
Let $h_k$, $P_k(i,j)$, $\phi_k(\Delta)$, $\pi_k[0]$, $\mu_k[0]$, $\Phi_k[0]$, and $\delta_k[0]$ be defined as in \cref{alg:init,alg:com_gen}. Then for $k\leq N-1$
\begin{equation}\label{eq:m_k_rewrite}
  m_k[0] + \sum_{i=1}^{k}\sum_{d = \sigma_i}^{\sigma_{i+1}-1}\bar{V}_i[d] = \delta_{k}[0]  - h_k\mu_{k+1}[0]
\end{equation}

\end{lemma}
\begin{proof}
Let $B_k[0] = z_k[0] + u_k[-\tau_k] + D_k[\sigma_k]$ and $B_k[i] = u_k[-(\tau_k-i)] + D_k[\sigma_{k}+i]$ for $1\leq i < \tau_k$.
We will prove the lemma by showing that for $1\leq \Delta \leq \tau_k$
\begin{multline}\label{eq:m_k_partial}
  m_{k-1}[0] + \sum_{i=1}^{k-1}\sum_{d = \sigma_i}^{\sigma_{i+1}-1}\bar{V}_i[d] + \sum_{d = \sigma_k}^{\sigma_k+\Delta-1}V_k[d] =\\
   (1-P_k(\Delta,1))\delta_{k-1}[0]  
   -\Big[h_{k-1}b_k(1-P(\Delta,1))+P_k(\Delta,\Delta)g_{k+1}(1)\ \ \ \mathclap{\prod_{j = \Delta+1}^{\tau_k}}\ \ \ g_k(j)\Big]\mu_{k+1}[0] \\
   - \sum_{d=0}^{\tau_k-1}B_k[d]\Big[
   P_k(\Delta,\min(d+1,\Delta))\prod_{j=\Delta+1}^{d+1}g_k(j)
   +(1-P_k(\Delta,1))h_{k-1}\prod_{j=2}^{d+1}g_k(j)
   \Big]
\end{multline}
More specifically we will show that
\begin{enumerate}
    \item \eqref{eq:m_k_partial} holds for $k=1$ and $\Delta = 1$.
    \item If \eqref{eq:m_k_partial} holds for some $k$ and $\Delta-1$ then it holds for $\Delta$ as well.
    \item If \eqref{eq:m_k_rewrite} holds for $k-1$ then \eqref{eq:m_k_partial} holds for $k$ and $\Delta = 1$.
    \item If \eqref{eq:m_k_partial} holds for $k$ and $\Delta = \tau_k$ then $\eqref{eq:m_k_rewrite}$ holds for $k$.
\end{enumerate}

For $k=1$ and $\Delta = 1$ the LHS of \eqref{eq:m_k_partial} is just $V_1[t]$. The RHS of \eqref{eq:m_k_partial} equals $-X_1(1)/\rho_1\cdot\mu_k[0]$ as $P_1(1,m) = X_1(1)/\rho_1$, $\delta_0=0$, and $h_0 = 0$.
And by \cref{lem:V_expr} the RHS is also equal to $V_1[t]$ since by \eqref{eq:delta1reduce} 
\[
  \xi_1[0] = \mu_k[0].
\]
 Thus \eqref{eq:m_k_partial} holds for $k=1$ and $\Delta = 1$.

 Applying \cref{lem:V_expr} on $V_k[\sigma_k+\Delta]$ for the LHS of \eqref{eq:m_k_partial} gives:
\begin{multline}\label{eq:ind_step}
    m_{k-1}[0] + \sum_{i=1}^{k-1}\sum_{d = \sigma_i}^{\sigma_{i+1}-1}\bar{V}_i[d] + \sum_{d = \sigma_k}^{\sigma_k+\Delta-1}V_k[d] = \\
    (1-\frac{X_k(\Delta)}{\rho_k})\Big( m_{k-1}[0] + \sum_{i=1}^{k-1}\sum_{d = \sigma_i}^{\sigma_{i+1}-1}\bar{V}_i[d] + \sum_{d = \sigma_k}^{\sigma_k+\Delta-2}V_k[d]  \Big)\\
     -\frac{X_k(\Delta)}{\rho_k}\Big[\sum_{d=0}^{\tau_k-1} B_k[d]\prod_{j=\Delta+1}^{d+1}g_k(j)  +  
     \mu_{k+1}[0]g_{k+1}(1)\prod_{j=\Delta+1}^{\tau_k} g_k(j)\Big]
\end{multline}
Now assume that \eqref{eq:m_k_partial} holds for $k$ and $\Delta-1$, we then show that \eqref{eq:m_k_partial} holds for $k$ and $\Delta$.
Using \eqref{eq:ind_step} gives for the coefficients for the different terms of the LHS for \eqref{eq:m_k_partial}  as follows.
For $\delta_{k-1}[0]$ we get
\[
  (1-\frac{X_k(\Delta)}{\rho_k})(1-P_k(\Delta-1,1)) = 1-P_k(\Delta,1).
\]
For the terms in front of $\mu_{k+1}[0]$ we get
\begin{multline*}
  -\Big(1-\frac{X_k(\Delta)}{\rho_k}\Big)\Big(h_{k-1}b_k
\big(1-P(\Delta-1,1) \big)
  +P_k(\Delta-1,\Delta-1)g_{k+1}(1)\prod_{j = \Delta}^{\tau_k} g_k[j]\Big) \\
   - \frac{X_k(\Delta)}{\rho_k}g_{k+1}(1)\prod_{j=\Delta+1}^{\tau_k} g_k(j) \\
  =-h_{k-1}b_k\big(1-P(\Delta,1)\big) - P_k(\Delta,\Delta)g_{k+1}(1)\prod_{j=\Delta+1}^{\tau_k} g_k(j),
\end{multline*}
and for the coefficient for $B[d]$,
\begin{multline*}
-(1-\frac{X_k(\Delta)}{\rho_k})\Big[P_k(\Delta-1,\min(d+1,\Delta-1))\prod_{j=\Delta}^{d+1}g_k(j) 
   + \\
   (1-P_k(\Delta-1,1))h_{k-1}\prod_{j=2}^{d+1}g_k(j)\Big]
- \frac{X_k(\Delta)}{\rho_k}\prod_{j=\Delta+1}^{d+1}g_k(j)\\
= - P_k(\Delta,\min(d+1,\Delta))\sum_{j=\Delta+1}^{d+1}g_k(j)
-(1-P_k(\Delta,1))h_{k-1}\prod_{j=2}^{d+1}g_k(j).
\end{multline*}
Thus \eqref{eq:m_k_partial} holds for $k$ and $\delta$ as well.

Assume that \eqref{eq:m_k_rewrite} holds for $k-1$. Then we can show that \eqref{eq:m_k_partial} holds for $k$ and $\Delta = 1$. 
Using that
\begin{equation}\label{eq:pi_k}
  \pi_k[0] = \sum_{d=0}^{\tau_k-1}\Big( B[d]\prod_{j=2}^{d+1}g_k(j)\Big),
\end{equation}
the RHS of \eqref{eq:m_k_partial} reduces to
\begin{equation*}
\big[1-P_k(1,1)\big]\delta_{k-1}[0] 
   -\big[h_{k-1}(1-P_k(1,1))+P_k(1,1)\big](\pi_k[0] + b_k\mu_{k+1}[0]) .
\end{equation*}
Using \eqref{eq:ind_step} with $\Delta = 1$, the definition for $P_k(1,1)$, and inserting  \eqref{eq:m_k_rewrite} gives that the LHS of \eqref{eq:m_k_partial} is equal to
\begin{equation*}
(1-P_k(1,1))\Big[\delta_{k-1}[0] - h_{k-1}\mu_{k}[0] \Big] 
    - P_k(1,1)\Big[\sum_{d=0}^{\tau_k-1} B[d]\prod_{\mathclap{j=\Delta+1}}^{d+1}g_k(j)  +  
     \mu_{k+1}[0]b_k\Big].
\end{equation*}
Using \eqref{eq:pi_k} and the definition for $\mu_k[0] = \pi_k[0] + b_k \mu_{k+1}[0]$ shows that the RHS and LHS are equal. And thus \eqref{eq:m_k_partial} hold for $k$ and $\Delta = 1$ if \eqref{eq:m_k_rewrite} holds for $k-1$.

Finally, we will show that if  \eqref{eq:m_k_partial} holds for $k$ and $\Delta = \tau_k$ then \eqref{eq:m_k_rewrite} holds for $k$. Using the definition for $h_k$  the RHS of \eqref{eq:m_k_partial} reduces to
\begin{multline*}
   (1-P_k(\tau_k,1))\delta_{k-1}[0]
   +h_k\mu_{k+1}[0] \\
   - \sum_{i=d}^{\tau_k}B_k[d]\Big[P_k(\tau_k,d+1) + \big(1-P_k(\tau_k,1)h_{k-1}\big)\prod_{j=2}^{d+1}g_k(j)\Big]
\end{multline*}
For $\Delta = \tau_k$ the LHS of \eqref{eq:m_k_rewrite} is equal to the LHS of \eqref{eq:m_k_partial} plus
\[
  z_k[0] +\sum_{d=1}^{\tau_k}u_k[-d] + \sum_{d=\sigma_k}^{\sigma_{k+1}-1} D_k[d] = \sum_{d=0}^{\tau_k-1}B_k[d].
\]
Thus it holds that the LHS of \eqref{eq:m_k_rewrite} is equal to
\begin{multline*}
   (1-P_k(\tau_k,1))\delta_{k-1}[0]
   +h_k\mu_{k+1}[0] \\
   + \sum_{i=d}^{\tau_k}B_k[d]\Big[(1-P_k(\tau_k,d+1)) - \big(1-P_k(\tau_k,1)h_{k-1}\big)\prod_{j=2}^{d+1}g_k(j)\Big].
\end{multline*}
Using the definition for $\phi_i(\Delta)$ in \cref{alg:init} and $\Phi_i$ and $\delta_k$ in \cref{alg:com_gen} shows that \eqref{eq:m_k_partial} gives \eqref{eq:m_k_rewrite} for $\Delta=\tau_k$, as
\[
  \sum_{i=d}^{\tau_k}B_k[d]\Big[(1-P_k(\tau_k,d+1)) - \big(1-P_k(\tau_k,1)h_{k-1}\big)\prod_{j=2}^{d+1}g_k(j)\Big] = \Phi_k[0].
\]

\end{proof}

We are now finally ready to prove the theorem, which follows from the previous lemmas.

\emph{Proof of \cref{thm:gen}:}
\Cref{lem:V_expr} with $\Delta = 1$ and \cref{lem:mVsum}  gives that
\begin{equation}\label{eq:opt_V}\begin{aligned}
  V_k[\sigma_k] &= -\frac{X_k(1)}{\rho_k}\Big[m_{k-1}[0]  +  \sum_{i=1}^{k-1}\sum_{d = \sigma_i}^{\sigma_{i+1}-1}\bar{V}_i[d]
   + \mu_k[0] \Big]\\
  & = -\frac{X_k(1)}{\rho_k}\Big[\delta_{k-1}[0]  + (1-h_k)\mu_k[0] \Big]
\end{aligned}\end{equation}
 from which the optimal $v_k[0] = \rho_k/r_k\cdot V_k[\sigma_k]$ as in \cref{alg:com_gen} follows.
Using \cref{lem:shifted_inv}-(iii), \cref{lem:mVsum}, and that $v_k[0] = \rho_k/r_k\cdot V_k[\sigma_k]$ gives that
\begin{multline*}
  u_{k-1}[0] = (1-\frac{\gamma_k}{q_k})(z_k[t]+u_k[-\tau_k]+D_k[0]) +d_k[0]-D_k[0]\\
  +V_k[\sigma_k]\Big(\frac{\rho_k}{r_k}-\frac{\gamma_k}{q_k} \Big) - \frac{\gamma_k}{q_k}\Big(\delta_{k-1}[0]-h_{k-1}\mu_k[0] \Big).
\end{multline*}
Inserting \eqref{eq:opt_V} gives that the optimal $u$ is as in \cref{alg:com_gen}. 

The results will hold for $t\neq 0$ as the problem has an infinite horizon and one can always change the variables so that current time is time zero.
\qed

\bibliographystyle{IEEEtran}
\bibliography{prod_everywhere}

\end{document}